\documentclass[12pt]{amsart}
\usepackage{epsfig}  		

\newtheorem{thm}{Theorem}[section]
\newtheorem{prop}[thm]{Proposition}

\theoremstyle{definition}
\newtheorem{defn}[thm]{Definition}
\newtheorem{ex}[thm]{Example}

\newtheorem{remark}[thm]{Remark}

\numberwithin{equation}{section}

\newcommand{\R}{\mathbb{R}}  
\newcommand{\B}{\mathcal{B}}
\newcommand{\N}{\mathbb{N}}
\newcommand{\Z}{\mathbb{Z}}

\begin{document}
\title{On The Intersection Algebra Of Principal Ideals}

\author{Sara Malec}
\address{Department of Mathematics, University of the Pacific, 
Stockton, CA 95207}
\email{smalec@pacific.edu}

\begin{abstract}
We study the finite generation of the intersection algebra of two principal ideals $I$ and $J$ in a unique factorization domain $R$. We provide an algorithm that produces a list of generators of this algebra over $R$. In the special case that $R$ is a polynomial ring, this algorithm has been implemented in the commutative algebra software system Macaulay2. A new class of algebras, called fan algebras, is introduced.
\end{abstract}

 \maketitle

\section{Introduction}
In this paper, we study the intersection of powers of two ideals in a commutative Noetherian ring. This is achieved by looking at the structure called the intersection algebra, a recent concept, which is associated to the two ideals.\\

The purpose of the paper is to study the finite generation of this algebra, and to show that it holds in the particular important case of principal ideals in a unique factorization domain (UFD). 
In the general case, not much is known about the intersection algebra, and there are many questions that can be asked. Various aspects of the intersection algebra have been studied by J. B. Fields in \cite{Fieldsthesis, Fields}. There, he proved several interesting things, including the finite generation of the intersection algebra of two monomial ideals in the power series ring over a field. He also studied the relationship between the finite generation of the intersection algebra and the polynomial behavior of a certain function involving lengths of Tors. It is interesting to note that this algebra is not always finitely generated, as shown by Fields. The finite generation of the intersection algebra has also appeared in the work of Ciuperc{\u{a}}, Enescu, and Spiroff  in \cite{Ciuperca} in the context of asymptotic growth powers of ideals.\\

We will start with the definition of the intersection algebra. Throughout this paper, $R$ will be a commutative Noetherian ring.\\

\begin{defn}Let $R$ be a ring with two ideals $I$ and $J$. Then the \emph{intersection algebra of $I$ and $J$} is $\mathcal{B}=\bigoplus_{r,s \in \mathbb{N}}I^r\cap J^s$. If we introduce two indexing variables $u$ and $v$, then $\mathcal{B}_R(I,J)=\sum_{r,s \in \mathbb{N}}I^r\cap J^su^rv^s\subseteq R[u,v]$. When $R,I$ and $J$ are clear from context, we will simply denote this as $\mathcal{B}$. We will often think of $\mathcal{B}$ as a subring of $R[u,v]$, where there is a natural $\mathbb{N}^2$-grading on monomials $b \in \mathcal{B}$ given by $\textrm{deg}(b)=(r,s)\in \mathbb{N}^2$. If this algebra is finitely generated over $R$, we say that $I$ and $J$ have \emph{finite intersection algebra}. 
\end{defn}


\begin{ex} If $R=k[x,y]$, $I=(x^2y)$ and $J=(xy^3)$, then an example of an element in $\mathcal{B}$ is $2+3x^5y^9u^2v^3+x^{10}y^{15}u^4v$ , since $2 \in I^0 \cap J^0=k$, $x^5y^9u^2v^3 \in I^2 \cap J^3u^2v^3=(x^4y^2) \cap (x^3y^9)u^2v^3=(x^4y^9)u^2v^3$, and $x^{10}y^{15}u^4v \in I^4\cap Ju^4v=(x^8y^4) \cap (xy^3)u^4v=(x^8y^4)u^4v$.
\end{ex}

We remark that the intersection algebra has connections to the double Rees algebra $R[Iu,Jv]$, although in practice they can be very different. This relationship is significant due to the importance of the Rees algebra, but the two objects behave differently. The source for the different behaviour lies in the obvious fact that the intersection $I^r \cap J^s$ is harder to predict than $I^rJ^s$ as $r$ and $s$ vary. These differences in behavior are of great interest and should be further explored.\\

In this paper, we produce an algorithm that gives a set of generators for the algebra for two principal ideals in a UFD, and we implement the algorithm in Macaulay2 for the case of principal monomial ideals in a polynomial ring over a field. \\

The finite generation of an $\N^2$-graded algebra, such as the intersection algebra, can be rephrased in the following way:

\begin{prop}Let $\B=\bigoplus_{r,s \in \N}\B_{r,s}$ be an $\N^2$-graded algebra over a ring $R$. Then $\B$ is finitely generated if and only if there exists an $N \in \mathbb{N}$ such that for every $r,s \in \mathbb{N}$ with $r,s>N$,

 \[\mathcal{B}_{r,s} = \sum_{\mathcal{I}}\prod_{j,k=0}^N\mathcal{B}_{j,k}^{i_{jk}},\]
where $\mathcal{I}$ is the set of all $N \times N$ matrices $I=(i_{jk})$ such that \[r=\sum_{j,k=0}^Nji_{jk}\textrm{ and } s=\sum_{j,k=0}^Nki_{jk}.\] \end{prop}
 
\begin{proof}If $\B$ is finitely generated over $R$, there exists an $N \in \mathbb{N}$ such that all the generators for $\B$ come from components $\B_{r,s}$ with $r,s \le N.$ Let $b \in B_{r,s}$ with $r,s>N$. So $b$ can be written as a polynomial in the generators with coefficients in $R$, in other words 
\[b\in\sum_{\textrm{finite}} \prod_{j,k=0}^N\B_{j,k}^{i_{jk}}.\] 
Also, the right hand side must be in the $(r,s)$-graded component of $\B$: i.e. the sum can only run over all $j,k$ such that $r=\sum ji_{jk}$ and $s=\sum ki_{jk}$ and $j$ and $k$ both go from 0 to $N$. So $b \in \B_{r,s}$ must be in the right hand side above as claimed. For the other inclusion, note again that the degrees match up:
\[\prod_{j,k=0}^N\mathcal{B}_{j,k}^{i_{jk}} \subset \B_{r,s}\]
 as long as $r=\sum ji_{jk}$ and $s=\sum ki_{jk}$, so obviously sums of such expressions are also included in $\B_{r,s}$. The other direction is obvious.
\end{proof}

\section{The UFD Case}
\begin{thm}If $R$ is a UFD and $I$ and $J$ are principal ideals, then $\mathcal{B}$ is finitely generated as an algebra over $R$.
\end{thm}

The proof for the finite generation of $\mathcal{B}$ relies heavily on semigroup theory. The following definitions and theorems provide the necessary framework for the proof.\\

\begin{defn} A \emph{semigroup} is a set together with a closed associative binary operation. A semigroup generalizes a monoid in that it need not contain an identity element. We call a semigroup an \emph{affine semigroup} if it isomorphic to a subgroup of $\mathbb{Z}^d$ for some $d$. An affine semigroup is called \emph{pointed} if it contains the identity, which is the only invertible element of the semigroup.\end{defn}

 \begin{defn}A \emph{polyhedral cone} in $\mathbb{R}^d$ is the intersection of finitely many closed linear half-spaces in $\mathbb{R}^d$, each of whose bounding hyperplanes contains the origin. Every polyhedral cone $C$ is finitely generated, i.e. there exist $\boldsymbol{c_1}, \ldots, \boldsymbol{c_r} \in \mathbb{R}^d$ with \[C=\{\lambda_1\boldsymbol{c_1}+\cdots+\lambda_r\boldsymbol{c_r}|\lambda_1, \ldots, \lambda_r \in \mathbb{R}_{\geq 0}\}.\] We call the cone $C$ \emph{rational} if $\boldsymbol{c_1}, \ldots, \boldsymbol{c_r}$ can be chosen to have rational coordinates, and $C$ is \emph{pointed} if $C \cap(-C)=\{\boldsymbol{0}\}$.\end{defn}
 
We present the following two important results without proof: the full proofs are contained in \cite{Miller} as Theorem 7.16 and 7.15, respectively.

 \begin{thm} (Gordan's Lemma) If $C$ is a rational cone in $\mathbb{R}^d$, then $C \cap A$ is an affine semigroup for any subgroup $A$ of $\mathbb{Z}^d$.\end{thm}
 
  \begin{thm} Any pointed affine semigroup $Q$ has a unique finite minimal generating set $\mathcal{H}_Q$.\end{thm}
 
 \begin{defn} Let $C$ be a rational pointed cone in $\mathbb{R}^d$, and let $Q=C \cap \mathbb{Z}^d$. Then $\mathcal{H}_Q$ is called the \emph{Hilbert Basis of the cone $C$}.\end{defn}
 
We will use these results to provide a list of generators for $\mathcal{B}$. Notice that for any two strings of numbers
\[\mathbf a=\{a_1, \ldots, a_n\}, \mathbf b=\{b_1, \ldots, b_n\}\textrm{ with }a_i,b_i \in \mathbb{N},\]
we can associate to them a fan of pointed, rational cones in $\mathbb{N}^2$. 

\begin{defn} We will call two such strings of numbers \emph{fan ordered} if 
\[\frac{a_i}{b_i} \geq \frac{a_{i+1}}{b_{i+1}}\textrm{ for all } i=1, \ldots, n.\] 
Assume $\mathbf{a}$ and $\mathbf{b}$ are fan ordered. Additionally, let $a_{n+1}=b_0=0$ and $a_0=b_{n+1}=1$. Then for all $i=0, \ldots, n$, let \[C_i=\{\lambda_1 (b_i,a_i )+\lambda_2(b_{i+1},a_{i+1})|\lambda_1, \lambda_2 \in \mathbb{R}_{\geq 0}\}.\] Let $\Sigma_{\mathbf{a}, \mathbf{b}}$ be the fan formed by these cones and their faces, and call it the \emph{fan of $\mathbf{a}$ and $\mathbf{b}$ in $\N^2$}. Hence \[\Sigma_{\mathbf{a},\mathbf{b}}=\{C_i|i=0, \ldots, n\}.\]\end{defn}
\begin{center}
\includegraphics{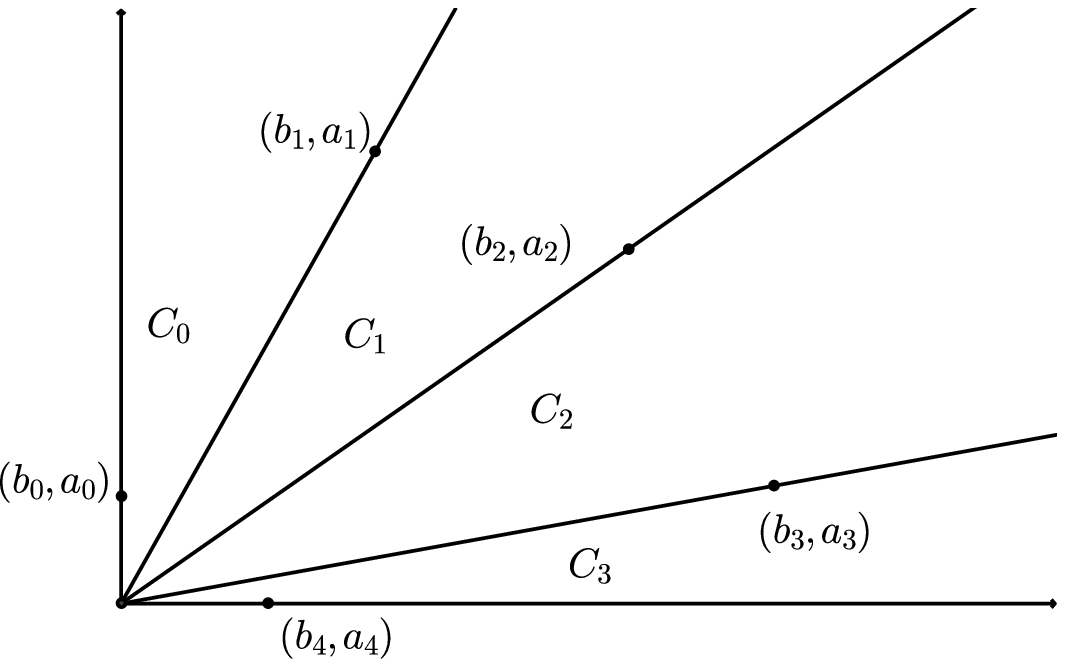}
\end{center}
Then, since each $C_i$ is a pointed rational cone, $Q_i=C_i \cap \mathbb{Z}^2$ has a Hilbert Basis, say 
\[\mathcal{H}_{Q_i}=\{(r_{i1},s_{i1}), \ldots, (r_{in_i},s_{in_i})\}.\]

Note that any $\Sigma_{\mathbf{a}, \mathbf{b}}$ partitions all of the first quadrant of $\R^2$ into cones, so the collection $\{Q_i|i=0, \ldots, n+1\}$ partitions all of $\N^2$ as well, so for any $(r,s) \in \N^2$, $(r,s) \in Q_i$ for some $i=0, \ldots, n+1$.\\

In this paper, we are studying the intersection algebra when $I$ and $J$ are principal, so the order of the exponents in their exponent vectors does not matter. In general, for any two strings of numbers $\mathbf{a}$ and $\mathbf{b}$, there is a unique way to rearrange them so that they are fan ordered. So a unique fan can be associated to any two vectors. For the purposes of this paper, we will assume without loss of generality that the exponent vectors are fan ordered.\\

 \begin{thm} \label{UFD} Let $R$ be a UFD with principal ideals $I=(p_1^{a_1}\cdots p_n^{a_n})$ and  $J=(p_1^{b_1}\cdots p_n^{b_n})$, where $p_i, i=0, \ldots, n$, and let $\Sigma_{\mathbf{a},\mathbf{b}}$ be the fan associated to $\mathbf{a}=(a_1, \ldots, a_n)$ and $\mathbf{b}=(b_1, \ldots, b_n)$. Then $\mathcal{B}$ is generated over $R$ by the set \[\{p_1^{a_1r_{ij}}\cdots p_i^{a_ir_{ij}}p_{i+1}^{b_{i+1}s_{ij}} \cdots p_n^{b_ns_{ij}}u^{r_{ij}}v^{s_{ij}}| i=0, \ldots, n, j=1, \ldots, n_i\},\] where $(r_{ij},s_{ij})$ run over the Hilbert basis for each $Q_i=C_i \cap \Z^2$ for every $C_i \in \Sigma_{\mathbf{a},\mathbf{b}}$.
 \end{thm}
 
  \begin{proof} Since $\mathcal{B}$ has a natural $\mathbb{N}^2$ grading, it is enough to consider only homogeneous monomials $b \in \mathcal{B}$ with $\textrm{deg}(b)=(r,s)$. Then $(r,s) \in Q_i=C_i \cap \Z^2$ for some $C_i \in \Sigma_{\mathbf{a},\mathbf{b}}$. In other words, $r,s \in \mathbb{N}^2$ and \[\frac{a_i}{b_i} \geq \frac{s}{r} \geq \frac{a_{i+1}}{b_{i+1}}.\] So $a_ir \geq b_is  $, and by the ordering on the $a_i$ and the $b_i$, $ a_jr\geq b_js$ for all $j<i$. Also, $a_{i+1}r\leq b_{i+1}s$, and again by the ordering, $a_jr\leq b_js$ for all $j>i$. So
\begin{align*}
b \in I^r \cap J^s u^rv^s &= (p_1^{a_1}\cdots p_n^{a_n})^r \cap (p_1^{b_1} \cdots p_n^{b_n})^su^rv^s\\
&=(p_1^{a_1r}\cdots p_i^{a_ir}\cdot p_{i+1}^{b_{i+1}s}\cdots p_n^{b_ns})u^rv^s
\end{align*}

So $b=f\cdot p_1^{a_1r}\cdots p_i^{a_ir}\cdot p_{i+1}^{b_{i+1}s}\cdot p_n^{b_ns}u^rv^s$ for some monomial $f \in R$.\\

Since $(r,s) \in Q_i$, the pair has a decomposition into a sum of Hilbert basis elements. So we have $(r,s)=\sum_{j=1}^{n_i} m_j(r_{ij},s_{ij})$ with $m_j \in \mathbb{N}$, and $r=\sum_{j=1}^{n_i}m_jr_{ij}$, $s=\sum_{j=1}^{n_i}m_js_{ij}$. Therefore
\begin{align*}b=&f(p_1^{a_1r} \cdots p_i^{a_ir}p_{i+1}^{b_{i+1}s}\cdots p_n^{b_ns}u^rv^s)\\
=&f\prod_{j=1}^{n_i}p_1^{m_j(a_1r_{ij})}\cdots p_i^{m_j(a_ir_{ij})}p_{i+1}^{m_j(b_{i+1}s_{ij})}\cdots p_n^{m_j(b_ns_{ij})}u^{m_j(r_{ij})}v^{m_j(s_{ij})}\\
=&f\prod_{j=1}^{n_i}(p_1^{a_1r_{ij}}\cdots p_i^{a_ir_{ij}}p_{i+1}^{b_{i+1}s_{ij}}\cdots p_n^{b_ns_{ij}}u^{r_{ij}}v^{s_{ij}})^{m_j}
\end{align*}

So $b$ is generated over $R$ by the given finite set as claimed.
\end{proof}

\begin{remark} This theorem extends and refines the main result in \cite{mythesis}
\end{remark}

\begin{remark}For any two ideals $I$ and $J$ in $R$ with $J \subset \sqrt{I}$, where $I$ is not nilpotent and $\cap_k I^k =(0)$, define $v_I(J,m)$ to be the largest integer $n$ such that $J^m \subseteq I^n$ and $w_J(I,n)$ to be the smallest $m$ such that $J^m \subseteq I^n$. The two sequences $\{v_I(J,m)/m\}_m$ and $\{w_J(I,n),n\}_n$ have limits $l_I(J)$ and $L_J(I)$, respectively. See \cite{Nagata, Samuel} for related work. 

Given two principal ideals $I$ and $J$ in a UFD $R$ whose radicals are equal (i.e. the factorizations of their generators use the same irreducible elements), our procedure to determine generators also shows that the vectors $(b_1,a_1)$ and $(b_n,a_n)$ are related to the pairs of points $(r,s)$ where $I^r \subseteq J^s$ (respectively $J^s \subseteq I^r$): notice that $C_0$ is the cone between the $y$-axis and the line through the origin with slope $a_0/b_0$, and for all $(r,s) \in C_0\cap \N^2$, $I^r \subseteq J^s$. Therefore  $l_J(I)=a_0/b_0$. Similarly, $C_n$, the cone between the $x$-axis and the line through the origin with slope $a_n/b_n$, contains all $(r,s) \in \N^2$ where $J^s \subseteq I^r$, so $l_J(I)=a_n/b_n$. Then, since $l_I(J)L_J(I)=1$, this gives that $L_J(I)=a_1/b_1$ and $L_I(J)=b_n/a_n$ as well. This agrees with the observations of Samuel and Nagata as mentioned in \cite{Ciuperca}. \\
\end{remark}

\section{The Polynomial Case}

In this section, we will show that in the special case where $R$ is a polynomial ring in finitely many variables over a field, then the intersection algebra of two principal monomial ideals is a semigroup ring whose generators can be algorithmically computed.\\

\begin{defn} Let $k$ be a field. The \emph{semigroup ring} $k[Q]$ of a semigroup $Q$ is the $k$-algebra with $k$-basis $\{t^a | a \in Q\}$ and multiplication defined by $t^a \cdot t^b=t^{a+b}$.\end{defn}

Note that when $F=\{f_1, \ldots, f_q\}$ is a collection of monomials in $R$, $k[F]$ is equal to the semigroup ring $k[Q]$, where $Q=\mathbb{N}\textrm{log}(f_1)+\cdots + \mathbb{N}\textrm{log}(f_q)$ is the subsemigroup of $\mathbb{N}^q$ generated by $\textrm{log}(F)$. It is easy to see that multiplying monomials in the semigroup ring amounts to adding exponent vectors in the semigroup, as in the following example:

We can consider $\mathcal{B}$ both as an $R$-algebra and as a $k$-algebra, and it is important to keep in mind which structure one is considering when proving results. While there are important distinctions between the two, finite generation as an algebra over $R$ is equivalent to finite generation as an algebra over $k$.

\begin{thm} Let $R$ be a ring that is finitely generated as an algebra over a field $k$. Then $\mathcal{B}$ is finitely generated as an algebra over $R$ if and only if it is finitely generated as an algebra over $k$.\end{thm}

\begin{proof} Let $\mathcal{B}$ be finitely generated over $k$. Then since $k \subset R$, $\mathcal{B}$ is automatically finitely generated over $R$. Now let $\mathcal{B}$ be finitely generated over $R$, say by elements $b_1, \ldots, b_n \in \mathcal{B}$. Then for any $b \in \mathcal{B}$, $b=\sum_{i=1}^q r_ib_i^{\alpha_i}$ with $r_i \in R$. But $R$ is finitely generated over $k$, say by elements $k_1, \ldots, k_m$, so $r_i=\sum_{j=1}^p a_{ij}k_j^{\beta_{ij}}$, with $a_{ij} \in k$. So $b=\sum_i^q (\sum_j^p a_{ij}k_j^{\beta_{ij}})b_i^{\alpha_i}$, and $\mathcal{B}$ is finitely generated as an algebra over $k$ by $\{b_1, \ldots, b_n, k_1, \ldots, k_m\}$.
\end{proof}

A few definitions are required before stating the main results of this section.

\begin{defn} Let $R=k[\boldsymbol{x}]=k[x_1,\ldots, x_n]$ be the polynomial ring over a field $k$ in $n$ variables. Let $F=\{f_1, \ldots, f_q\}$ be a finite set of distinct monomials in $R$ such that $f_i \neq 1$ for all $i$. The \emph{monomial subring spanned by $F$} is the $k$-subalgebra \[k[F]=k[f_1, \ldots, f_q] \subset R.\]\end{defn}

\begin{defn} For $c \in \mathbb{N}^n$, we set $\boldsymbol{x}^c=x_1^{c_1} \cdots x_n^{c_n}$. Let $f$ be a monomial in $R$. The exponent vector of $f=\boldsymbol{x}^\alpha$ is denoted by $\textrm{log}(f)=\alpha \in \mathbb{N}^n$. If $F$ is a collection of monomials in $R$, $\textrm{log}(F)$ denotes the set of exponent vectors of the monomials in $F$.\end{defn}

\begin{thm} If $R$ is a polynomial ring in $n$ variables over $k$, and $I$ and $J$ are ideals generated by monomials (i.e. monic products of variables) in $R$, then $\mathcal{B}$ is a semigroup ring.\end{thm}

\begin{proof} Since $I$ and $J$ are monomial ideals, $I^r \cap J^s$ is as well for all $r$ and $s$. So each $(r,s)$ component of $\mathcal{B}$ is generated by monomials, therefore $\mathcal{B}$ is a subring of $k[x_1, \ldots, x_n, u,v]$ generated over $k$ by a list of monomials $\{b_i|i \in \Lambda\}$. Let  $Q$ be the semigroup generated by $\{\textrm{log}(b_i)|i \in \Lambda\}$. Then $\mathcal{B}=k[Q]$, and $\mathcal{B}$ is a semigroup ring over $k$.
 \end{proof}

 \begin{thm} Let $I=(x_1^{a_1}\cdots x_n^{a_n})$ and  $J=(x_1^{b_1}\cdots x_n^{b_n})$ be principal ideals in $R=k[x_1, \ldots, x_n]$, and let $\Sigma_{\mathbf{a},\mathbf{b}}$ be the fan associated to $\mathbf{a}=(a_1, \ldots, a_n)$ and $\mathbf{b}=(b_1, \ldots, b_n)$. Let 
 \[Q_i=C_i \cap \Z^2 \textrm{ for every } C_i \in \Sigma_{\mathbf{a},\mathbf{b}}\] 
 and $\mathcal{H}_{Q_i}$ be its Hilbert basis of cardinality $n_i$ for all $i=0, \ldots, n$.
 Further, let $Q$ be the subsemigroup in $\N^2$ generated by 
 \[\{(a_1r_{ij}, \ldots, a_ir_{ij},b_{i+1}s_{ij},\ldots,b_ns_{ij},r_{ij},s_{ij})| i=0, \ldots, n, j=1, \ldots, n_i\} \cup \log(x_1, \ldots, x_n),\]
 where $(r_{ij},s_{ij})\in \mathcal{H}_{Q_i}$ for every $i=0, \ldots n, j=1, \ldots, n_i$. Then $\mathcal{B}=k[Q]$.
 \end{thm}

 \begin{proof} 
 
Since $R$ is a UFD, by Thm \ref{UFD}, $\mathcal{B}$ is generated over $R$ by \[\{x_1^{a_1r_{ij}}\cdots x_i^{a_ir_{ij}}x_{i+1}^{b_{i+1}s_{ij}} \cdots x_n^{b_ns_{ij}}u^{r_{ij}}v^{s_{ij}}| i=0, \ldots, n, j=1, \ldots, n_i\}.\] Then, since $R$ is generated as an algebra over $k$ by $x_1, \ldots, x_n$, it follows that 
$\mathcal{B} \subset k[x_1, \ldots, x_n,u,v]$ is generated as an algebra over $k$ by the set \[P =\{x_1, \ldots, x_n, x_1^{a_1r_{ij}}\cdots x_i^{a_ir_{ij}}x_{i+1}^{b_{i+1}s_{ij}} \cdots x_n^{b_ns_{ij}}u^{r_{ij}}v^{s_{ij}}| i=0, \ldots, n, j=1, \ldots, n_i\}.\]
This is a set of monomials in $k[x_1, \ldots, x_n,u,v]$. Now note that therefore
\begin{align*}
\log(P) =&\{(a_1r_{ij}, \ldots, a_ir_{ij},b_{i+1}s_{ij},\ldots,b_ns_{ij},r_{ij},s_{ij})| i=0, \ldots, n, j=1, \ldots, n_i\} \\
& \qquad \cup \textrm{log}(x_1, \ldots, x_n).\\
\end{align*}
In conclusion, $\log(P)=Q$ and hence $\B=k[Q]$.
 
\end{proof}

\begin{ex} Let $I=(x^5y^2)$ and $J=(x^2y^3)$. Then $a_1=5, a_2=2$ and $b_1=2, b_2=3$, and $5/2 \geq 2/3$. Then we have the following cones:
\begin{align*}
C_0&=\{\lambda_1(0,1)+\lambda_2(2,5)|\lambda_i \in \mathbb{R}_{\geq 0}\}\\
C_1&=\{\lambda_1(2,5)+\lambda_2(3,2)|\lambda_i \in \mathbb{R}_{\geq 0}\}\\
C_2&=\{\lambda_1(3,2)+\lambda_2(1,0)|\lambda_i \in \mathbb{R}_{\geq 0}\}\\
\end{align*}

$C_0$ is the wedge of the first quadrant between the $y$-axis and the vector $(2,5)$, $C_1$ is the wedge between $(2,5)$ and $(3,2)$, and $C_3$ is the wedge between $(3,2)$ and $(1,0)$. It is easy to see that this fan fills the entire first quadrant. Intersecting these cones with $\mathbb{Z}^2$ is equivalent to only considering the integer lattice points in these cones.\\

The Hilbert Basis of $Q_0=C_0 \cap \mathbb{Z}^2$ is $\{(0,1),(1,3),(2,5)\}$, and their corresponding monomials in $\mathcal{B}$ are given by the generators of $\mathcal{B}_{r,s}$ for each $(r,s)$:

\begin{align*}
(0,1)&: (I^0 \cap J^1)u=(x^2y^3)v - \textrm{generator is }x^2y^3v\\
(1,3)&: (I^1 \cap J^3)uv^3=((x^{5}y^{2}) \cap (x^6y^9))uv^3=(x^{6}y^{9})uv^3 - \textrm{generator is }x^{6}y^{9}uv^3\\
(2,5)&: (I^2 \cap J^5)u^2v^5=((x^{10}y^{4}) \cap (x^{10}y^{15}))u^2v^5=(x^{10}y^{15})u^2v^5 - \textrm{generator is }x^{10}y^{15}u^2v^5\\
\end{align*}
Notice that all the generator monomials are of the form $x^{b_1s}y^{b_2s}u^rv^s$, with $b_1=2, b_2=3$, and $(r,s)$ is a Hilbert Basis element, as shown earlier.

The Hilbert Basis of $Q_1$ is $\{(1,1),(1,2),(3,2),(2,5)\}$. In the same way as above, their monomials are $x^5y^3uv, x^5y^{6}uv^2, x^{15}y^{6}u^3v^2, x^{10}y^{15}u^2v^5$, all of which have the form $x^{a_1r}y^{b_2s}u^rv^s$ with $a_1=5, b_2=3$ and $(r,s)$ a basis element.\\

Lastly, the Hilbert Basis of $Q_2$ is $\{(1,0), (2,1), (3,2)\}$, which gives rise to generators $x^5y^2u, x^{10}y^4u^2, x^{15}y^{6}u^3v^2$, all of which look like $x^{a_1r}y^{a_2r}u^rv^s$ with $a_1=5, a_2=2$.\\

Notice there are a few redundant generators in this list: those arise from lattice points that lie on the boundaries of the cones. So $\mathcal{B}$ is generated over $R$ by \[\{x^5y^2u, x^{10}y^4u^2, x^{15}y^{6}u^3v^2, x^5y^3uv, x^5y^{6}uv^2, x^2y^3v, x^{6}y^{9}uv^3, x^{10}y^{15}u^2v^5\}.\]\end{ex}

Using this technique, we have written a program in Macaulay2 that will provide the list of generators of $\mathcal{B}$ for any $I$ and $J$. First it fan orders the exponent vectors, then finds the Hilbert Basis for each cone that arises from those vectors. Finally, it computes the corresponding monomial for each basis element. The code is below:

\begin{verbatim}
loadPackage "Polyhedra"
--function to get a list of exponent vectors from an ideal I
expList=(I) ->(
     flatten exponents first flatten entries gens I
)

algGens=(I,J)->(
     B:=(expList(J))_(positions(expList(J),i->i!=0));
     A:=(expList(I))_(positions(expList(J),i->i!=0));
     L:=sort apply(A,B,(i,j)->i/j);   
     C:=flatten {0,apply(L,i->numerator i),1};
     D:=flatten {1, apply(L,i->denominator i),0};
     M:=matrix{C,D};
     G:=unique flatten apply (#C-1, i-> hilbertBasis 
     		(posHull submatrix(M,{i,i+1})));
     S:=ring I[u,v];
     flatten apply(#G,i->((first flatten entries gens 
     		intersect(I^(G#i_(1,0)),J^(G#i_(0,0)))))*u^(G#i_(1,0))*v^(G#i_(0,0)))
)

\end{verbatim}

\section{Fan Algebras}
The intersection algebra is in fact a specific case of a more general class of algebras that can be naturally associated to a fan of cones. We will call such objects fan algebras, and the first result in this section shows that they are finitely generated. First a definition:

\begin{defn}Given a fan of cones $\Sigma_{\mathbf{a}, \mathbf{b}}$, a function $f:\N^2 \rightarrow \N$ is called \emph{fan-linear} if it is nonnegative and linear on each subgroup $Q_i=C_i \cap \Z^2$ for each $C_i \in \Sigma_{\mathbf{a}, \mathbf{b}}$, and subadditive on all of $\N^2$, i.e.
\[f(r,s)+f(r',s') \geq f(r+r', s+s') \textrm{ for all }(r,s), (r',s') \in \N^2.\]
In other words, $f(r,s)$ is a piecewise linear function where 
\[f(r,s) = g_{i}(r,s)  \textrm{ when } (r,s) \in C_i\cap \N^2 \textrm{ for each } i=0, \ldots n, \textrm{ and each }g_i\textrm{ is linear on }C_i\cap \N^2.\] 
Note that each piece of $f$ agrees on the faces of the cones, that is $g_i=g_j$ for every $(r,s) \in C_i \cap C_j\cap \N^2$.
\end{defn}

\begin{ex} Let $\mathbf{a}=\{1\}=\mathbf{b}$, so $\Sigma_{\mathbf{a}, \mathbf{b}}$ is the fan defined by 
\begin{align*}C_0&=\{\lambda_1(0,1)+\lambda_2(1,1)|\lambda_i \in \mathbb{R}_{\geq 0}\}\\
C_1&=\{\lambda_1(1,1)+\lambda_2(1,0)|\lambda_i \in \mathbb{R}_{\geq 0}\},\end{align*}
and set $Q_i=C_i \cap \Z^2$.
Also let \[f=\left\{ \begin{array}{ll}
g_0(r,s)=r+2s & \textrm{if $(r,s) \in Q_0$}\\ g_1(r,s)=2r+s & \textrm{if $(r,s) \in Q_1$}\end{array} \right.\]
Then $f$ is a fan-linear function. It is clearly nonnegative and linear on both $Q_0$ and $Q_1$. The function is also subadditive on all of $\N^2$: Let $(r,s) \in Q_0$ and $(r',s') \in Q_1$, and say that $(r+r', s+s') \in Q_0$. Then 
\begin{align*}f(r,s)+f(r',s')&=g_0(r,s)+g_1(r',s')=r+2s+2r'+s\\
f(r+r',s+s')&=g_0(r+r',s+s')=r+r'+2(s+s')\\
\end{align*}
Comparing the two, we see that \[f(r,s)+f(r',s')\geq f(r+r',s+s') \textrm{ whenever }r+2s+2r'+s' \geq r+r'+2(s+s'),\] or equivalently when $r' \geq s'$. But that is true, since $(r', s') \in Q_1$. The proof for $(r+r', s+s') \in Q_1$ is similar. The two pieces of $f$ also agree on the boundary between $Q_0$ and $Q_1$, since the intersection of $Q_0$ and $Q_1$ is the ray in $\N^2$ where $r=s$, and \[g_0(r,r)=3r=g_1(r,r).\] So $f$ is a fan-linear function.
\end{ex}

\begin{thm}\label{gen}Let $I_1 \ldots, I_n$ be ideals in a domain $R$ and $\Sigma_{\mathbf{a}, \mathbf{b}}$ be a fan of cones in $\N^2$. Let $f_1, \ldots, f_n$ be fan-linear functions. Then the algebra
\[\B=\bigoplus_{r,s}I_1^{f_1(r,s)}\cdots I_n^{f_n(r,s)}u^rv^s\] is finitely generated.\end{thm}

\begin{proof} First notice that the subadditivity of the functions $f_i$ guarantees that $\B$ is a subalgebra of $R[u,v]$ with the natural grading. Since $\mathcal{B}$ has a natural $\mathbb{N}^2$-grading, it is enough to consider only homogeneous monomials $b \in \mathcal{B}$ with $\textrm{deg}(b)=(r,s)$. Then $(r,s) \in Q_i=C_i \cap \Z^2$ for some $C_i \in \Sigma_{\mathbf{a},\mathbf{b}}$. Since $Q_i$ is a pointed rational cone, it has a Hilbert basis
\[H_{Q_i}=\{(r_{i1},s_{i1}), \ldots, (r_{in_i},s_{in_i})\}.\]
So we can write \[(r,s)=\sum_{j=1}^{n_i}m_j(r_{ij},s_{ij}).\] Then, since each $f_k$ is nonnegative and linear on $Q_i$, we have
\[f_k(r,s)=\sum_{j=1}^{n_i}m_jf_k(r_{ij},s_{ij})\]
 for each $k=1, \ldots, n$. Since $R$ is Noetherian, for each $i$, there exists a finite set $\Lambda_{i,j,k} \subset R$ such that \[I_k^{f_k(r_{ij},s_{ij})}=(x_k|x_k \in \Lambda_{i,j,k}).\] So
\begin{align*}
b \in \B_{r,s}&=I_1^{f_1(r,s)}\cdots I_n^{f_n(r,s)}u^rv^s\\
&=I_1^{\sum_{j=1}^{n_i}m_jf_1(r_{ij},s_{ij})} \cdots I_n^{\sum_{j=1}^{n_i}m_jf_n(r_{ij},s_{ij})}u^{\sum_{j=1}^{n_i}m_jr_{ij}}v^{\sum_{j=1}^{n_i}m_js_{ij}}\\
&=I_1^{m_1f_1(r_{i1},s_{i1})}\cdots I_1^{m_{n_i}f_1(r_{in_i},s_{in_i})} \cdots I_n^{m_1f_n(r_{i1},s_{i1})}\cdots I_n^{m_{n_i}f_n(r_{in_i},s_{in_i})}\\
&\qquad \qquad u^{m_1r_{i1}}\cdots u^{m_{n_i}r_{in_i}}v^{m_1s_{i1}}\cdots v^{m_{n_i}s_{in_i}}\\
&=\left(I_1^{f_1(r_{i1},s_{i1})}\cdots I_n^{f_n(r_{i1},s_{i1})}u^{r_{i1}}v^{s_{i1}}\right)^{m_1}\cdots\left(I_1^{f_1(r_{in_i},s_{in_i})}\cdots I_n^{f_n(r_{in_i},s_{in_i})}u^{r_{in_i}}v^{s_{in_i}}\right)^{m_{n_i}}.
\end{align*}
So $\B$ is generated as an algebra over $R$ by the set
\[\{x_1 \cdots x_nu^{r_{ij}}v^{s_{ij}}|(r_{ij},s_{ij}) \in \mathcal{H}_{Q_i}, x_k \in \Lambda_{i,j,k}\}.\]
\end{proof}

This result justifies the following definition.

\begin{defn} Given ideals $I_1, \ldots, I_n$ in a domain $R$, $\Sigma_{\mathbf{a}, \mathbf{b}}$ a fan of cones in $\N^2$, and $f_1, \ldots, f_n$ are fan-linear functions, we define
 \[\B(\Sigma_{\mathbf{a},\mathbf{b}},f)=\bigoplus_{r,s}I_1^{f_1(r,s)}\cdots I_n^{f_n(r,s)}u^rv^s\]
 to be the \emph{fan algebra of $f$ on $\Sigma_{\mathbf{a},\mathbf{b}}$}, where $f=(f_1, \ldots, f_n)$.
\end{defn}

\begin{remark}The intersection algebra of two principal ideals $I=(p_1^{a_1}\cdots p_n^{a_n})$ and $J=(p_1^{b_1}\cdots p_n^{b_n})$ in a UFD is a special case of a fan algebra. Let $I_i=(p_i)$ and $f_i=\max(ra_i,sb_i)$ for each $i=1, \ldots, n$, and define the fan $\Sigma_{\mathbf{a},\mathbf{b}}$ to be the fan associated to $\mathbf{a}=(a_1, \ldots, a_n)$ and $\mathbf{b}=(b_1, \ldots, b_n)$. Then 
\[\B(I,J)=\bigoplus_{r,s}(p_1)^{\max(ra_1,sb_1)}\cdots (p_n)^{\max(ra_n,sb_n)}u^rv^s.\]
This is a fan algebra since the max function is fan-linear: it's subadditive on all of $\N^2$, and linear and nonnegative on each cone, since the faces of each cone in $\Sigma_{\mathbf{a},\mathbf{b}}$ are defined by lines through the origin with slopes $a_i/b_i$ for each $i=0, \ldots, n$. So, as in the proof of Theorem \ref{UFD}, for any pair $(r,s) \in Q_i=C_i\cap \Z^2$ for every $C_i \in \Sigma_{\mathbf{a},\mathbf{b}}$, we have that
 \[\frac{a_i}{b_i} \geq \frac{s}{r} \geq \frac{a_{i+1}}{b_{i+1}}.\] 
 So $a_ir \geq b_is  $, and by the ordering on the $a_i$ and the $b_i$, $ a_jr\geq b_js$ for all $j<i$. Also, $a_{i+1}r\leq b_{i+1}s$, and again by the ordering, $a_jr\leq b_js$ for all $j>i$. Since $f_k=\max(ra_k,sb_k)$ for all $k=1, \ldots, n$, we have that
 \[f_k=ra_k \textrm{ for all } k \leq i \textrm{ and } f_k=sb_k \textrm{ for all } k > i.\]
 So each $f_k$ is linear on each cone, and the above theorem applies.  \\
\end{remark}

It is important to note that the intersection algebra is not always Noetherian. One such example, given in \cite{Fieldsthesis}, is constructed by taking an ideal $P$ in $R$ such that the algebra \[R \oplus P^{(1)} \oplus P^{(2)} \oplus \cdots\] is not finitely generated. Then, it is shown that there exists an $f \in R$ such that $(P^a:f^a)=P^{(a)}$ for all $a$. It follows that the intersection algebra of $P$ and $(f)$ is not finitely generated.

One important question that should be considered is what conditions on $f$ and $P$ are necessary to ensure that the intersection algebra of $f$ and $P$ is Noetherian.\\

\begin{prop}Let $R$ be a standard $\N^2$-graded ring with maximal ideal $\mathfrak{m}$ and $f$ a homogeneous element in $ \mathfrak{m}$. Then \[\B=\B_R(f,\mathfrak{m})=\bigoplus_{r,s \in \mathbb{N}^2}(f)^r\cap \mathfrak{m}^su^rv^s\] is finitely generated as an $R$-algebra.
\end{prop}
\begin{proof}Say $\deg f=a$, and let $x \in (f)^r \cap \mathfrak{m}^s$. Then $x=f^r \cdot y \in \mathfrak{m}^s$, so $y \in (\mathfrak{m}^s:f^r)$. Decompose $y$ into its homogeneous pieces, $y=y_0+ \cdots+ y_m$. Then $f^r(y_0+\cdots+y_m) \in \mathfrak{m}^s$, so \[ra+\deg y_i \geq s \textrm{ for all }i=0, \ldots, n,\] or, equivalently $\deg y_i \geq s-ra$. Therefore $y \in \mathfrak{m}^{s-ra}$, and so \[(f)^r \cap (\mathfrak{m}^s)=f^r \cdot \mathfrak{m}^{s-ra}.\]

Then
 \[\B=\sum_{r,s \in \N^2} f^r \cdot \mathfrak{m}^{s-ra}u^rv^s=\sum_{r,s \in \N^2}\mathfrak{m}^{s-ra}(fu)^rv^s.\]
Let 
\[\tilde{\mathcal{B}}=\sum_{r,s \in \N^2}\mathfrak{m}^{s-ra}w^rv^s,\]
and $\varphi: \tilde{\B} \rightarrow \B$ be the map that sends $w$ to $fu$ and is the identity on $R$ and $v$. This map is obviously surjective, therefore, if $\tilde{\mathcal{B}}$ is finitely generated, so is $\mathcal{B}$. But $\tilde{\B}$ is a fan algebra with $I_1=\mathfrak{m}$ and
 \[f_1(r,s) = \left\{ \begin{array}{ll}
s-ra & \textrm{if $s/r \geq a$}\\ 
 0 & \textrm{if $s/r<a$}\\ \end{array} \right., \] which is certainly fan-linear on the fan formed by two cones \[C_0=\{\lambda_1(0,1)+\lambda_2(1,a))|\lambda_i \in \R_{\geq 0}\}\textrm{ and } C_1=\{\lambda_1(1,a)+\lambda_2(1,0)|\lambda_i \in \R_{\geq 0}\},\]
 since $C_0$ contains the collection of all $(r,s) \in \N^2$ where $s/r \geq a$ and $C_1$ contains all $(r,s) \in \N$ where $s/r <a$.

So by Theorem \ref{gen}, if we define $\Lambda_{ij} \subset R$ to be a finite subset where $\mathfrak{m}^{s_{ij}-r_{ij}a}=(x|x \in \Lambda_{ij})$, then $\tilde{\B}$ is generated over $R$ by the set
\[\{xu^{r_{0j}}v^{s_{0j}}|(r_{0j},s_{0j}) \in \mathcal{H}_{Q_0}, x\in \Lambda_{ij}\}\cup \{u^{r_{1j}}v^{s_{1j}}|(r_{1j},s_{1j}) \in \mathcal{H}_{Q_1}\} ,\]
and therefore $\B$ is finitely generated as an $R$-algebra.
\end{proof}

\begin{remark}We are grateful to Mel Hochster, who noticed that the above proof works the same in the case where $R$ is regular local by replacing the degree of $f$ with its order. When $(R,\mathfrak{m})$ is a regular local ring, the order defines a valuation on $R$ because $gr_{\mathfrak{m}}(R)$ is a polynomial ring over $R/\mathfrak{m}$ (and hence a domain). Therefore, when $(R,\mathfrak{m})$ is a regular local ring and $f \in R$, $B(f, \mathfrak{m})$ is a finitely generated $R$-algebra.
\end{remark}

\subsection*{Acknowledgements} The author thanks her advisor, Florian Enescu, for many fruitful discussions, as well as Yongwei Yao for his useful suggestions.
\bibliographystyle{plain}
\bibliography{bibliography}

\end{document}